\documentclass[12pt,oneside,a4paper]{amsart}
\usepackage{microtype}
\usepackage{amssymb}
\usepackage{tikz}

\usepackage[T1]{fontenc}
\usepackage{newtxtext,newtxmath}

\renewcommand{\epsilon}{\varepsilon}
\renewcommand{\phi}{\varphi}

\newtheorem{theorem}{Theorem}

\makeatletter
\renewcommand\subsection{\@startsection{subsection}{2}%
  \z@{.5\linespacing\@plus.7\linespacing}{.1\linespacing}%
  {\normalfont\scshape}}
\makeatother

\title[On the geometric interpretation of A3w]{A remark on the geometric interpretation of the A3w condition from optimal transport}
\author{Cale Rankin}
\thanks{This research is supported by ARC DP 200101084 and the Fields Institute for Research in Mathematical Sciences.}
\address{The Fields Institute for Research in Mathematical Sciences}
\email{cale.rankin@utoronto.ca\\cale.rankin@gmail.com}
\begin{document}
\maketitle
\begin{abstract}
We provide a geometric interpretation of the well known A3w condition for regularity of optimal transport maps. 
\end{abstract}
\section{Introduction}
\label{sec:introduction}

In optimal transport a condition known as A3w is necessary for regularity of the optimal transport map. Here we provide a geometric interpretation of A3w.  We'll use freely the notation from \cite{MR2188047}. Let $c \in C^2(\mathbf{R}^n \times \mathbf{R}^n)$ satisfy A1 and A2 (see \S \ref{sec:proof-results}). Keeping in mind the prototypical case $c(x,y) = |x-y|^2$, we fix $x_0,y_0 \in \mathbf{R}^n$ and perform a linear transformation so $c_{xy}(x_0,y_0) = -I$. Define coordinates
\begin{align}
  \label{eq:q} q(x)&:= -c_y(x,y_0),\\
  \label{eq:p} p(y)&:= -c_x(x_0,y),
\end{align}
and the inverse transformations by $x(q),y(p)$. Write $c(q,p) = c(x(q),y(p))$ and let $q_0=q(x_0)$ and $p_0=p(y_0)$. 
We prove A3w is satisfied if and only if whenever these transformations are performed
\begin{align}
  \nonumber
 (q-q_0)\cdot (p-p_0) \geq 0 \implies &c(q,p)+c(q_0,p_0) \leq c(q,p_0)+c(q_0,p). 
\end{align}
Heuristically, A3w implies when $q-q_0$ ``points in the same direction'' as $p-p_0$ it is cheaper to transport $q$ to $p$ and $q_0$ to $p_0$ than the alternative $q$ to $p_0$ and $q_0$ to $p$. Thus, A3w implies compatibility between directions in the cost-convex geometry and the cost of transport.

 A3w first appeared (in a stronger form) in \cite{MR2188047}. It was weakened in \cite{MR2512204} and a new interpretation given in \cite{MR2506751}. The impetus for the above interpretation is Lemma 2.1 in \cite{MR3419751}. Our result can also be realised by a particular choice of c-convex function in the unpublished preprint \cite{trudiner-wang-preprint}. \\

\textit{Acknowledgements. }  My thanks to Jiakun Liu and Robert McCann for helpful comments and discussion. 
 
\section{Proof of result}
\label{sec:proof-results}
 Let $c \in C^2(\mathbf{R}^n \times \mathbf{R}^n)$ satisfy following the well known conditions\\
\textbf{A1. } For each $x_0,y_0 \in \mathbf{R}^n$ the following mappings are injective
\begin{align*}
  x \mapsto c_y(x,y_0), \quad\text{and} \quad  y \mapsto c_x(x_0,y).
\end{align*}
\textbf{A2. } For each $x_0,y_0 \in \mathbf{R}^n$ we have $\det c_{i,j}(x_0,y_0) \neq 0$.

Here, and throughout, subscripts before a comma denote differentiation with respect to the first variable, subscripts after a comma denote differentiation with respect to the second variable.

By A1 we define on $\mathcal{U}:= \{(x,c_x(x,y)); x,y \in \mathbf{R}^n\}$ a mapping $Y:\mathcal{U}\rightarrow \mathbf{R}^n$ by
\[ c_x(x,Y(x,p)) = p. \]
The A3w condition, usually expressed with fourth derivatives but written here as in \cite{MR4322877}, is the following.\\
\textbf{A3w. } Fix $x$. The function
\[ p \mapsto c_{ij}(x,Y(x,p))\xi_i\xi_j,\]
is concave along line segments orthogonal to $\xi$.

To verify A3w it suffices to verify midpoint concavity, that is  whenever $\xi\cdot \eta = 0$ there holds
\begin{equation}
  \label{eq:mid-point}
   0 \geq [c_{ij}(x,Y(x,p+\eta)) - 2c_{ij}(x,Y(x,p)) + c_{ij}(x,Y(x,p-\eta))]\xi_i\xi_j. 
\end{equation}

Finally, we recall $A \subset \mathbf{R}^n$ is called $c$-convex with respect to $y_0$ provided $c_y(A,y_0)$ is convex.  When A3w is satisfied and $y,y_0 \in \mathbf{R}^n$ are given the section $\{x \in \mathbf{R}^n; c(x,y) > c(x,y_0)\}$ is $c$-convex with respect to $y_0$ \cite{MR4322877}.

Now fix $(x_0,p_0) \in \mathcal{U}$ and $y_0 = Y(x_0,p_0)$. To simplify the proof we assume $x_0,y_0,q_0,p_0 = 0$. Up to an affine transformation (replace $y$ with $\tilde{y}:=-c_{xy}(0,0)y$) we assume $c_{xy}(0,0) = -I$. Note with $q,p$ as defined in \eqref{eq:q}, \eqref{eq:p}, this implies $\frac{\partial q}{\partial x}(0) = I$.  Put
\begin{align*}
  \tilde{c}(x,y) &:= c(x,y) - c(x,0) - c(0,y) + c(0,0),\\
  \overline{c}(q,p) &:= \tilde{c}(x(q),y(p)).                        \end{align*}
 
\begin{theorem}\label{thm:main}
  The A3w condition is satisfied if and only if whenever the above transformations are applied the following implication holds
  \begin{equation}
    \label{eq:imp}
       q \cdot p \geq 0 \implies \overline{c}(q,p) \leq 0. 
  \end{equation}
\end{theorem}
\begin{proof}
  Observe by a Taylor series 
  \begin{equation}
    \label{eq:exp}
     \overline{c}(q,p) = -(q \cdot p) + \overline{c}_{ij}(\tau q,p)q_iq_j,
  \end{equation}
  for some $\tau \in (0,1)$. First, assume A3w and let $q \cdot p > 0$. By \eqref{eq:exp} we have $\overline{c}(-tq,p) > 0 > \overline{c}(tq,p)$ for $t>0$ sufficiently small. If $\overline{c}(q,p) > 0$ then the $c$-convexity (in our coordinates, convexity) of the section
  \[ \{ q \ ; \ \overline{c}(q,p) > \overline{c}(q,0) = 0 \}, \]
  is violated. By continuity $\overline{c}(q,p) \leq 0$ whenever $q \cdot p \geq 0$.

  In the other direction, take nonzero $q$ with $q \cdot p = 0$ and small $t$. By  \eqref{eq:imp} and \eqref{eq:exp}
  \[ 0 \geq \overline{c}(t q,p)/t^2 = \overline{c}_{ij}(t \tau q,p)q_iq_j. \]
  This inequality also holds with $-p$. Moreover $\overline{c}_{ij}(t \tau q, 0) = 0$. 
  Thus
  \[ 0 \geq [\overline{c}_{ij}(t \tau q,p) - 2\overline{c}_{ij}(t \tau q , 0) + \overline{c}_{ij}(t \tau q, -p)]q_iq_j.\]
  Sending $t \rightarrow 0$ and returning to our original coordinates we obtain \eqref{eq:mid-point}. 
\end{proof}

\noindent \textbf{Remarks. } (1) On a Riemannian manifold with $c(x,y) = d(x,y)^2$, for $d$ the distance function, Loeper \cite{MR2506751} proved A3w implies nonnegative sectional curvature. Our result expedites his proof.  Let $x_0=y_0 \in M$ and $u,v \in T_{x_0}M$ satisfy $u\cdot v = 0$ with $x = \exp_{x_0}(tu),y=\exp_{x_0}(tv)$. Working in a sufficiently small local coordinate chart our previous proof implies if A3w is satisfied
  \begin{equation}
    \label{eq:1}
     d(x,y)^2 \leq d(x_0,y)^2+d(x_0,x)^2 = 2t.
  \end{equation}
  The sectional curvature in the plane generated by $u,v$ is the $\kappa$ satisfying 
  \begin{equation}
 d(\exp_{x_0}(tu),\exp_{x_0}(tv)) = \sqrt{2}t\big(1-\frac{\kappa}{12}t^2+O(t^3)\big) \text{ as }t \rightarrow 0, \label{eq:2}
\end{equation}
whereby comparison with \eqref{eq:1} proves the result (see \cite[eq. 1]{MR3544918} for \eqref{eq:2}). We note Loeper proved his result using an infinitesimal version of \eqref{eq:1}. \\

\bibliographystyle{plain}
\bibliography{../bibliography.bib}
\end{document}